\documentclass[a4paper,11pt]{article}
\usepackage{mathrsfs}
\usepackage{amsmath}                    
\usepackage{amssymb}                    
\usepackage{amsthm}                     
\usepackage{amsfonts}                   
\usepackage{color}
\usepackage{graphics,graphicx}
\usepackage{enumerate}
\usepackage{epsfig}
\usepackage{subcaption}

\usepackage{multirow}
\usepackage[english]{algorithm}
\usepackage{dsfont}

\newtheorem{theorem}{Theorem}[section]

\newtheorem{proposition}[theorem]{Proposition}
\newtheorem{lemma}[theorem]{Lemma}
\newtheorem{remark}[theorem]{Remark}
\newtheorem{corollary}[theorem]{Corollary}
\newtheorem{ass}[theorem]{Assumption}

\numberwithin{equation}{section}
\def\s{\varsigma}

\renewcommand{\epsilon}{\varepsilon}
\newcommand{\eps}{\varepsilon}

\newcommand{\R}{\mathbb{R}}

\newcommand{\N}{\mathbb{N}}

\newcommand{\T}{\mathbb{T}}

\DeclareMathOperator{\sign}{sign}

\def\RR{{\mathbb R}}
\def\NN{{\mathbb N}}
\def\ZZ{{\mathbb Z}}

\def\be{\begin{equation}}
\def\ee{\end{equation}}

\title{
Highly-oscillatory problems with time-dependent vanishing frequency
}

\author{ 
Philippe Chartier\textsuperscript{1}, Mohammed Lemou\textsuperscript{2}, Florian M\'ehats\textsuperscript{3} and Gilles Vilmart\textsuperscript{4}
}

\begin{document}
\footnotetext[1]{
Univ Rennes, INRIA, IRMAR - UMR 6625, F-35000 Rennes, France.
Philippe.Chartier@inria.fr}
\footnotetext[2]{
Univ Rennes, CNRS, IRMAR - UMR 6625, F-35000 Rennes, France.
Mohammed.Lemou@univ-rennes1.fr}
\footnotetext[3]{
Univ Rennes, IRMAR - UMR 6625, F-35000 Rennes, France.
Florian.Mehats@univ-rennes1.fr}
\footnotetext[4]{
Universit\'e de Gen\`eve, Section de math\'ematiques, 2-4 rue du Li\`evre, CP 64, CH-1211 Gen\`eve 4, Switzerland, Gilles.Vilmart@unige.ch}

\maketitle

\begin{abstract}
In the analysis  of highly-oscillatory evolution problems, it is commonly assumed that a single frequency is present and that it is either constant or, at least, bounded from below by a strictly positive constant uniformly in time.  
Allowing for the possibility that the frequency actually depends on time and vanishes at some instants introduces additional difficulties from both the asymptotic analysis and numerical simulation points of view.  This work is a first step towards the resolution of these difficulties. In particular, we show that it is still possible in this situation to infer the asymptotic behaviour of the solution at the price of more intricate computations and we derive a second order uniformly accurate numerical method. 

\smallskip

 \end{abstract}

\section{Introduction}
\subsection{Context}
In this paper, we are concerned with oscillatory  differential equations whose frequency of oscillation depends on time. More precisely, we consider systems of differential equations (for some $T>0$) of the form 
\begin{eqnarray} \label{eq-Ueps}
\dot U^\eps(t) =\frac{\gamma(t)}{\eps}AU^{\eps}(t) + f\Big(U^{\eps}(t)\Big)   \in \RR^d,  \quad U^\eps(0)=U_0 \in \RR^d, \quad 0 \leq t \leq T,
\end{eqnarray} 
where the dot stands for the time derivative, the matrix $A\in\R^{d\times d}$ is supposed to be diagonalizable and to have all its eigenvalues in $i\ZZ$ (equivalently $\exp(2\pi A)=I$), where the function $f$ is assumed to be sufficiently smooth, where the parameter $\eps$  lies in $(0,1]$, and where the real-valued function $\gamma$ is assumed to be continuous on  $[0,+\infty)$.  However, {\sl the  main novel assumption} in this article is that the {\sl function $\gamma$ vanishes at some instant $t_0$}, or more precisely, that there exists (a unique)  $t_0\in [0,T]$ such that $ \gamma(t_0)=0$. 

As a related recent work, we mention the study \cite{ArD18} for the
uniformly accurate approximation of the stationary Schr\"odinger
equation in the presence of turning points which are spatial points used in quantum tunnelling models and where the spatial oscillatory frequency vanishes (analogously to our assumption $\gamma(t_0)=0$). While only the linear case is studied in \cite{ArD18} based on a
Wentzel-Kramers-Brillouin expansion, an additional difficulty is that
the Schr\"odinger equation solution blows up in the neighbourhood of
such a turning point asymptotically in the semi-classical limit where
$\eps\rightarrow 0$.

Our goal is to investigate problem (\ref{eq-Ueps}) under these new circumstances, from both the asymptotic analysis (when $\eps\to 0$) and the numerical approximation viewpoints. For the sake of simplicity in this introductory paper, we assume that $\gamma(t)$ is of the form
\footnote{Note that applying an analytic time-transformation to \eqref{eq-Ueps} allows to consider 
more general analytic functions $\gamma(t)$ and our analysis is not restricted to the polynomial case.}
$$
\exists p\in \NN^*, \quad \forall t\geq 0, \quad 
\gamma(t) = (p+1)(t-t_0)^p.
$$  
We emphasize that this situation is not covered by the  standard theory of averaging as considered e.g. in \cite{P69,SV85,HLW06,CMSS10,CMSS15,CLM17}, and that recent numerical approaches \cite{CLM13,CCMSS11,CCLM15,CLMV18} are ineffective. 
All techniques therein indeed rely fundamentally on the assumption that $\gamma(t)\geq \gamma_0$ uniformly in time, for some constant $\gamma_0 >0$, and can not be transposed to the context under consideration here. 

\subsection{Formulation as a periodic non-autonomous problem and main results} 
Upon defining $ u^\eps(t) = \exp\left( - \frac{(t-t_0)^{p+1}}{\eps}A\right) U^\eps(t)$, the original equation \eqref{eq-Ueps} may be rewritten 
  \begin{eqnarray} \label{eq-ueps}
\dot u^\eps(t) = F \left(\frac{(t-t_0)^{p+1}}{\eps}, u^{\eps}(t)\right),    \quad u^\eps(0)= u_0^\eps:=\exp \left(-\frac{(-t_0)^{p+1}}{\eps} A \right) \, U_0, 
\end{eqnarray}
where, for $\nu=\pm 1$,  $F(\theta, u)=   e^{- \theta  A } f( e^{\theta A} u)$ is 2$\pi$-periodic w.r.t. $\theta$ and smooth in $(\theta,u)$. 
We make the following assumption, which is naturally satisfied if $f$ is assumed to be locally Lipschitz continuous:
\begin{ass}
\label{hyp1}
There exist $T>0$ and $M>0$ such that for all $0 <\eps\leq 1$, \eqref{eq-ueps} has a unique solution on $[0,T]$, bounded by $M$, uniformly w.r.t. $\eps$. 
\end{ass}
In the sequel, $C$ will denote a {\it generic constant} that only depends on $t_0$ and on the bounds of $\partial^\alpha_2 F$, $\alpha=0,1,2,3$, on the set  $\{(\theta,u), \theta \in \T, |u|\leq 2 M\}$, where $\T=[0,2\pi]$.

The aim of this work is now twofold. 
On the one hand, we show that, under mild and standard assumptions, an averaged equation for \eqref{eq-ueps} of the form 
\begin{equation}
\label{averagedmodelu}
\forall t \in [0,T], \quad \dot{\underline{u}}^\eps(t)= \left\langle F\right\rangle(\underline{u}^\eps(t)),\qquad \underline{u}^\eps(0)=u_0^\eps
\end{equation}
persists.\footnote{Note that here as in the sequel, we denote the average of a function $\omega: \T \mapsto \R^d$ by
$$\left\langle \omega\right\rangle=\frac{1}{2\pi}\int_0^{2\pi}\omega(\theta)d\theta.
$$}
More precisely, we have the following theorem (see the proof in Section \ref{sect:averaged}), which can be refined with the next-order asymptotic term (see Section \ref{subsec-next}).
\begin{theorem} \label{thm:main}
Consider the  solutions $u^\eps(t),\underline u^\eps(t)$ of problems \eqref{eq-ueps}, \eqref{averagedmodelu}, respectively, on the time interval $[0,T]$. 
Then, there exists $\eps_0>0$ such that for all $\eps\in]0,\eps_0[$, and all $t \in [0,T]$,
\begin{equation} \label{eq:uepsubareps}
|u^\eps(t) - \underline{u}^\eps(t)| \leq C \eps^{\frac{1}{p+1}}.
\end{equation}
\end{theorem}
Note that the bound
$|u^\eps(t) - \underline{u}^\eps(t)| \leq C \eps$ obtained in the classical case of a constant frequency  is degraded to \eqref{eq:uepsubareps}. 

On the other hand, we construct in the case $p=1$ a second-order {\it uniformly accurate} scheme for the approximation of $u^\eps$, that is to say a method for which the error and the computational cost remain independent of the value of $\eps$ (for more details on uniformly accurate methods, refer for instance to \cite{CCLM15,CLMV18}). 

\section{Averaging results}
\label{sec-asympt}
We introduce the following function $\Gamma:[0,T] \rightarrow [0,S]$ with $S=(T-t_0)^{p+1}+t_0^{p+1}$,
$$
\Gamma(t) := \int_{0}^t  |\gamma(\xi)|d\xi=t_0^{p+1}+ (t-t_0) \, |t-t_0|^{p}=t_0^{p+1}+ \mu_t \, (t-t_0)^{p+1}, \quad \mu_t=\sign(t-t_0)^p = \pm 1, $$ 
and notice right away that $\Gamma$ is invertible with inverse $\Gamma^{-1}:[0,S] \rightarrow [0,T]$ given by
$$
\Gamma^{-1}(s) = s_0^{\frac{1}{p+1}} + \sign(s-s_0) \; |s-s_0|^{\frac{1}{p+1}}, \quad s_0=t_0^{p+1}.
$$ 
\begin{figure}[H]
\centering
\includegraphics[width=0.33\textwidth]{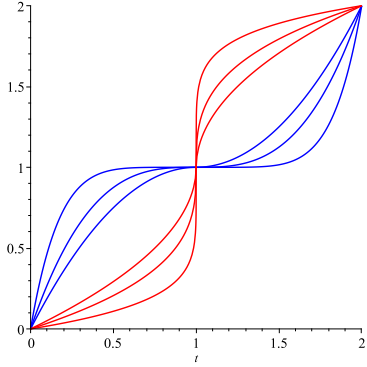}
\caption{The functions $\Gamma$ (in blue) and $\Gamma^{-1}$ (in red)  with $t_0=1$ and $T=2$ for $p=1,2,5$.}\label{fig1}
\end{figure}
Let us now consider $v^\eps(s)=u^\eps(t)$,  which, for $s \neq s_0$,  satisfies 
\begin{align} \label{eq:diffs}
\frac{d}{ds} v^\eps(s)
&= \frac{1}{\Gamma' \circ \Gamma^{-1}(s)} \, \dot u^\eps  \left(\Gamma^{-1}(s)\right) 
= \frac{1}{(p+1)|s-s_0|^{\frac{p}{p+1}}} \, F_{\mu_s} \left(\frac{s-s_0}{\eps}, v^{\eps}(s)\right)
\end{align}
with initial condition $v^\eps(0)=v_0^\eps:=u_0^\eps$, $\mu_s=\sign(s-s_0)^p$ and $F_\nu(\theta,u):=F(\nu \theta, u)$. 
As an immediate consequence of Assumption \ref{hyp1}, equation \eqref{eq:diffs} has a unique solution on $[0,S]$,  bounded by $M$ uniformly in $0 < \eps \leq 1$.

In this section, our aim is to show that there exists an averaged model for \eqref{eq:diffs} of the form 
\begin{equation}
\label{averagedmodel}
\forall s \in [0,S], \quad \dot{\underline{v}}^\eps(s)=\frac{1}{(p+1) |s-s_0|^\frac{p}{p+1}}\left\langle F\right\rangle(\underline{v}^\eps(s)),\qquad \underline{v}^\eps(0)=v_0^\eps,
\end{equation}
and then construct the first term of the asymptotic expansion of $v^\eps$ (see Section \ref{subsec-next}). 
Note that, despite the singularity at $s=s_0$ of the right-hand side of \eqref{averagedmodel}, its integral formulation clearly indicates the existence of a {\it continuous} solution on $[0,S]$.   

\subsection{Preliminaries}
Let us  introduce the following 2$\pi$-periodic zero-average functions 
$$G_{\nu}(\theta,u)=\int_0^\theta(F_{\nu}(\sigma,u)-\langle F\rangle(u))d\sigma-\left\langle \int_0^s(F_{\nu}(\sigma,u)-\langle F\rangle(u))d\sigma\right\rangle,$$
and
$$H_{\nu}(\theta,u)=\int_0^\theta G_{\nu}(\sigma,u)d\sigma-\left\langle \int_0^sG_{\nu}(\sigma,u)d\sigma\right\rangle.$$
It is clear that these functions and their derivatives in $u$ are uniformly bounded: for $\nu=\pm 1$, $|u|\leq 2M$, $v\in \RR^d$ and $s\in\RR$, we have
\begin{equation}\label{est0}|G_{\nu}(s,u)|+|H_{\nu}(s,u)|\leq C,\qquad |\partial_2 G_{\nu}(s,u)v|+ |\partial_2 H_{\nu}(s,u)v|\leq C|v|,\end{equation}
\begin{equation}\label{est0bis}|\partial^2_2 G_{\nu}(s,u)(v,v)|+ |\partial^2_2 H_{\nu}(s,u)(v,v)|\leq C|v|^2.\end{equation}
For all $\nu=\pm 1$, we eventually define the function
\begin{equation}\label{defOmega}
\forall u\in \RR^d, \forall s \in \RR_+, \quad \Omega_{\nu}(s,u)=\int_{s}^{+\infty}\frac{1}{\sigma^\frac{p}{p+1}}(F_\nu(\sigma,u)-\langle F\rangle(u))d\sigma.
\end{equation}
The following two technical lemmas will be useful all along this article.
\begin{lemma}
\label{lemOmega}
The function $\Omega_\nu$ is well-defined for all $s \in \RR_+$ and $u\in \RR^d$. Moreover, for all $u$ satisfying $|u|\leq 2M$, all $\nu=\pm1$, all $s \geq 0$ and all $v \in \RR^d$, we have the estimates
\begin{equation}\label{est1}
|\Omega_\nu (s,u)|\leq C, \quad 
|\partial_2 \Omega_\nu(s,u) v|\leq C|v|,\quad |\partial^2_2 \Omega_\nu(s,u) (v,v)|\leq C|v|^2.
\end{equation}
Restricting to strictly positive values of $s$, i.e. $s>0$, we have furthermore
\begin{equation}\label{est2}
\left|\Omega_\nu(s,u)\right|\leq \frac{C}{s^{\frac{p}{p+1}}},\quad \left|\partial_2\Omega_\nu(s,u)v\right|\leq \frac{C |v|}{s^{\frac{p}{p+1}}},
\end{equation}
and 
\begin{equation}\label{est3}
\left|\Omega_\nu(s,u)+\frac{G_\nu(s,u)}{s^\frac{p}{p+1}}\right|\leq \frac{C}{s^{1+\frac{p}{p+1}}},
\quad \left| \partial_2 \Omega_\nu(s,u)v+\frac{\partial_2 G_\nu(s,u)v}{s^{\frac{p}{p+1}}}\right|\leq \frac{C|v|}{s^{1+\frac{p}{p+1}}}.
\end{equation}
\end{lemma}
\begin{proof}
We  only prove the results for $\Omega_\nu$ as their adaptation to $\partial_2 \Omega_\nu$ and $\partial^2_2 \Omega_\nu$ is immediate. An integration by parts yields
$$\Omega_\nu(s,u)=-\frac{G_\nu(s,u)}{s^{\frac{p}{p+1}}}+\frac{p}{p+1}\int_s^{+\infty}\frac{1}{\sigma^{1+\frac{p}{p+1}}}G_\nu(\sigma,u)d\sigma,$$
where, from \eqref{est0}, the last integral is convergent and bounded by $\frac{C}{s^\frac{p}{p+1}}$. This yields the well-posedness of $\Omega_\nu$ for all $s>0$ and \eqref{est2}. We now simply remark that for all $s\geq 0$
$$\Omega_\nu(s,u)=\int_s^{1}\frac{1}{\sigma^{\frac{p}{p+1}}}(F_\nu(\sigma,u)-\langle F\rangle(u))d\sigma+\Omega_\nu(1,u).$$
This gives the well-posedness for $s=0$ and \eqref{est1} can be deduced from \eqref{est2} written for $s=1$.
A second integration by parts then gives
$$\Omega_\nu(s,u)=-\frac{G_\nu(s,u)}{s^{\frac{p}{p+1}}}-\frac{p}{p+1}\frac{H_\nu(s,u)}{s^{1+\frac{p}{p+1}}}+\frac{p}{p+1}\left(1+\frac{p}{p+1}\right)\int_s^{+\infty}\frac{1}{\sigma^{2+\frac{p}{p+1}}}H_\nu(\sigma,u)d\sigma.$$
Previous integral is bounded by $\frac{C}{s^{1+\frac{p}{p+1}}}$ owing to \eqref{est0} and this yields \eqref{est3}. 
\end{proof}
\begin{remark} \label{rem:splusun}
 Since $\left(\frac{1+s}{s}\right)^{\frac{p}{p+1}} \leq 2$ for $s \geq 1$, estimates \eqref{est1} and \eqref{est2} also imply for instance that for all $s \geq 0$,
\begin{equation*}
\left|\Omega_\nu(s,u)\right|\leq \frac{C}{(1+s)^{\frac{p}{p+1}}} \quad \mbox{ and } \quad \left|\partial_2 \Omega_\nu(s,u)v\right|\leq \frac{C |v|}{(1+s)^{\frac{p}{p+1}}}.
\end{equation*}
\end{remark}
\begin{lemma} \label{lem:integralI}
For a given $p \in \N^*$, consider two smooth functions $\phi, \psi: \T \times \R^d \rightarrow\R^d$ satisfying the estimates
\begin{align} \label{eq:ass}
|\psi(\sigma,u)| \leq C \quad \mbox{ and } \quad \left| \phi(\sigma, u) + \frac{\psi(\sigma,u)}{(1+\sigma)^\frac{p}{p+1}}\right| \leq \frac{C}{(1+\sigma)^{1+\frac{p}{p+1}}},
\end{align}
for all $\theta \in \T$ and all $|u| \leq M$
and define, for $0 \leq a \leq b \leq S$, the integral
\begin{align} \label{eq:I}
\mathcal I(a,b)=\frac{1}{p+1} \int_{a}^b \frac{1}{|\sigma-s_0|^\frac{p}{p+1}} \phi \left(\frac{|\sigma-s_0|}{\eps},v^\eps(\sigma)\right) d\sigma
\end{align}
where $v^{\eps}$ satisfies \eqref{eq:diffs}. Then, if $p=1$, we have 
\begin{align} 
\forall b \in [0,s_0], \quad & \mathcal I(0,b)=\frac{\sqrt{\eps}}{2} \, \log\left(\frac{\eps+s_0-b}{s_0+\eps}\right) \, \left\langle \psi \right\rangle (v^\eps(b))+ {\cal O}(\sqrt{\eps}), \label{eq:I1}\\
\forall b \in [s_0,S], \quad & \mathcal I(s_0,b)=\frac{\sqrt{\eps}}{2} \, \log\left(\frac{\eps}{b-s_0+\eps}\right) \, \left\langle \psi \right\rangle (v^\eps(s_0)) + {\cal O}(\sqrt{\eps}), \label{eq:I1bis}
\end{align}
while if $p \geq 2$, we have the estimate
\begin{align} \label{eq:Ip}
\forall 0 \leq a \leq b \leq S, \quad |\mathcal I(a,b)| \leq C \eps^\frac{1}{p+1}. 
\end{align}
\end{lemma}
\begin{proof}
Consider $0 \leq b \leq s_0$. A change of variables allows to write $\mathcal I(0,b)$ as
$$
\mathcal I(0,b) = \frac{\eps^{\frac{1}{p+1}}}{p+1} \int^{\frac{s_0}{\eps}}_{\frac{s_0-b}{\eps}} \frac{1}{\sigma^\frac{p}{p+1}}\phi \left(\sigma,v^\eps(s_0-\eps \sigma)\right) d\sigma.
$$
Now, we split  $(p+1) \eps^{\frac{-1}{p+1}} \mathcal I(0,b)=\mathcal J_2+\mathcal J_3+\mathcal J_4-\mathcal J_1$ into the sum of the four terms 
\begin{align*}
\mathcal J_2=&\int^{\frac{s_0}{\eps}}_{\frac{s_0-b}{\eps}}  \left(\frac{1}{(1+\sigma)^{\frac{2p}{p+1}}} - \frac{1}{(\sigma(1+\sigma))^{\frac{p}{p+1}}}  \right) \langle \psi \rangle \left(v^\eps(s_0-\eps \sigma)\right)  d\sigma,\\
\mathcal J_3=&\int^{\frac{s_0}{\eps}}_{\frac{s_0-b}{\eps}}  \frac{1}{ (\sigma(1+\sigma))^{\frac{p}{p+1}}} \left( \langle \psi \rangle -\psi \right) \left(\sigma,v^\eps(s_0-\eps \sigma)\right)  d\sigma,\\
\mathcal J_4=& \int^{\frac{s_0}{\eps}}_{\frac{s_0-b}{\eps}}  \frac{1}{\sigma^\frac{p}{p+1}} r(\sigma,v^{\eps}(s_0-\eps \sigma))d\sigma, \quad 
\mathcal J_1=\int^{\frac{s_0}{\eps}}_{\frac{s_0-b}{\eps}}  \frac{1}{(1+\sigma)^{\frac{2p}{p+1}}}  \langle \psi \rangle \left(v^\eps(s_0-\eps \sigma)\right) d\sigma,
\end{align*}
where we have denoted $r(\sigma,u) = \phi(\sigma, u) + \frac{\psi(\sigma,u)}{(1+\sigma)^\frac{p}{p+1}}$. 
Owing to assumption \eqref{eq:ass} and 
$$
\frac{1}{(1+\sigma)^{\frac{2p}{p+1}}} - \frac{1}{\sigma^\frac{p}{p+1} (1+\sigma)^{\frac{p}{p+1}}} \sim-\frac{p}{p+1}\frac{1}{\sigma^{\frac{3p+1}{p+1}}},
$$ 
integrals $\mathcal J_2$ and $\mathcal J_4$  are absolutely convergent and bounded. As for $\mathcal J_3$, we use the relation
\begin{align*}
-\frac{(\psi-\langle \psi \rangle)(\sigma,v^\eps(s_0-\eps \sigma))}{\sigma^\frac{p}{p+1}(1+\sigma)^\frac{p}{p+1}}  
=\frac{d}{d\sigma}\left(\kappa\left(\sigma,v^\eps(s_0-\eps \sigma)\right)\right)+\frac{\eps^\frac{1}{p+1}}{(p+1)\sigma^\frac{p}{p+1}}(\partial_2 \kappa \, F_{-\mu}  )\left(\sigma,v^\eps(s_0-\eps \sigma)\right)
\end{align*}
where we have taken equation \eqref{eq:diffs} into account with $\mu_s=\mu=(-1)^{p}$ and 
$$
\kappa(s,u)=\int_s^{+\infty}\frac{(\psi-\langle \psi \rangle)(\sigma,u)}{\sigma^\frac{p}{p+1}(1+\sigma)^{\frac{p}{p+1}}}d\sigma,
$$
in order to write $\mathcal J_3$ as 
\begin{align*}
\mathcal J_3=
\kappa\left(\frac{s_0}{\eps},v^\eps(0)\right)-\kappa\left(\frac{s_0-b}{\eps},v^\eps(b)\right)+\frac{\eps^\frac{1}{p+1}}{(p+1)}\int^{\frac{s_0}{\eps}}_{\frac{s_0-b}{\eps}} \frac{1}{\sigma^\frac{p}{p+1}}(\partial_2 \kappa \, F_{-\mu} )\left(\sigma,v^\eps(s_0-\eps \sigma)\right)d\sigma
\end{align*}
from which we may prove that $\mathcal J_3$ is bounded (note indeed that $\partial_2 \kappa \, F_{-\mu}$ is bounded). For $p>1$ it is clear that $\mathcal J_1$ is bounded owing to \eqref{eq:ass} and finally, that $\mathcal I(0,b)$ is bounded. The contribution of $\mathcal J_1$ for $p=1$ is more intricate and requires to be decomposed  as follows
\begin{align*}
\mathcal J_1&=\int^{\frac{s_0}{\eps}}_{\frac{s_0-b}{\eps}} \frac{1}{1+\sigma} \langle \psi \rangle (v^\eps(b))d\sigma + \int^{\frac{s_0}{\eps}}_{\frac{s_0-b}{\eps}}\frac{1}{1+\sigma} \Big(\langle \psi \rangle(v^\eps(s_0-\eps \sigma)) - \langle \psi \rangle(v^\eps(b))\Big)d\sigma\\
&=\log\left(\frac{s_0+\eps}{\eps+s_0-b}\right) \langle \psi \rangle(v^\eps(b))+\int^{\frac{s_0}{\eps}}_{\frac{s_0-b}{\eps}}\frac{1}{1+\sigma}\left(\langle \psi \rangle(v^\eps(s_0-\eps\sigma))-\langle \psi \rangle(v^\eps(b))\right)d\sigma.
\end{align*}
To estimate the second term, we use \eqref{eq:diffs} and $s_0-\eps \sigma \leq b \leq s_0$ to get 
\begin{align*}
\Big|\left[\langle \psi \rangle(v^\eps(\tau))\right]^{s_0-\eps \sigma}_b \Big|&\leq \left|\int_{s_0-\eps \sigma}^{b}\frac{1}{2\sqrt{s_0-\tau}}\left(\langle \partial_2 \psi \rangle \, F_{\mu} \right) (\frac{\tau-s_0}{\eps}, v^\eps(\tau)) d\tau\right|\leq C \sqrt{\eps \sigma}
\end{align*}
so that 
$$\left|\int^{\frac{s_0}{\eps}}_{\frac{s_0-b}{\eps}}\frac{\left(\langle \psi \rangle(v^\eps(s_0-\eps\sigma))-\langle \psi \rangle(v^\eps(b))\right)}{1+\sigma}d\sigma\right| \leq C \sqrt{\eps} \int^{\frac{s_0}{\eps}}_{0} \frac{\sqrt{\sigma}}{(1+\sigma)} d \sigma \leq C \sqrt{s_0}. 
$$
We finally obtain that
$$\mathcal I(0,b)=\frac{\sqrt{\eps}}{2} \, \log\left(\frac{\eps+s_0-b}{s_0+\eps}\right) \,  \langle \psi \rangle(v^\eps(b))+\mathcal O(\sqrt{\eps}).$$
{\it Mutatis mutandis}, a similar conclusion holds true for the case $a=s_0$ and $b \geq s_0$ as can be seen by writing the new value of $\mathcal J_1$ as 
\begin{align*}
\int_{0}^{\frac{b-s_0}{\eps}}  \frac{\langle \psi \rangle (v^\eps(s_0))+\langle \psi \rangle(v^\eps(s_0+\eps \sigma)) - \langle \psi \rangle(v^\eps(s_0))}{1+\sigma}  =\log\left(1+\frac{b-s_0}{\eps}\right) \langle \psi \rangle(v^\eps(s_0))+\mathcal O(1).
\end{align*}
\end{proof}
\subsection{The averaged model}
\label{sect:averaged}
We are now in position to state the first averaging estimate, from which Theorem \ref{thm:main} follows by considering the change of variable $\Gamma$. 
\begin{proposition} 
\label{prop1}Let $v^\eps$ be the solution of  problem \eqref{eq:diffs} on $[0,S]$, under Assumption \ref{hyp1}. Then, for all $0 < \eps < \eps_0$ where $\eps_0$ depends only on bounds on the derivatives of $F$,  the solution $\underline{v}^{\eps}$ of the averaged model \eqref{averagedmodel} exists on $[0,S]$ and one has
\begin{equation}\label{cv}
\forall s\in [0,S], \quad |v^\eps(s)-\underline{v}^\eps(s)|\leq C\,\eps^\frac{1}{p+1}.
\end{equation}
\end{proposition}
\begin{proof} The integral formulation of equation (\ref{eq:diffs}) reads
\begin{equation}\label{duhamel1}
v^\eps(s)=v_0^\eps+\frac{1}{p+1}\int_{0}^s\frac{1}{|\sigma-s_0|^\frac{p}{p+1}}\left\langle F\right\rangle (v^\eps(\sigma))d\sigma+R^\eps(s),
\end{equation}
where (with  $\mu_\sigma = \sign(\sigma-s_0)^p$)
\begin{align}
\label{defR}
R^\eps(s)&=\frac{1}{p+1}\int_{0}^s\frac{1}{|\sigma-s_0|^\frac{p}{p+1}}\left(F_{\mu_\sigma}\left(\frac{\sigma-s_0}{\eps},v^\eps(\sigma)\right)-\left\langle F\right\rangle (v^\eps(\sigma))\right)d\sigma,
\end{align} 
which is well-defined for all $s \in [0,S]$. From (\ref{defOmega}) with $\s= \sign(\sigma-s_0)$, $\sigma \neq s_0$, we have 
\begin{align*}
\frac{d}{d\sigma}\Omega_\nu\left(\frac{|\sigma-s_0|}{\eps},v^{\eps}(\sigma)\right) &= \frac{\s}{\eps} (\partial_1 \Omega_\nu) \left(\frac{|\sigma-s_0|}{\eps},v^{\eps}(\sigma)\right) + (\partial_2 \Omega_\nu) \left(\frac{|\sigma-s_0|}{\eps},v^{\eps}(\sigma)\right) \; \dot v^{\eps}(\sigma) \\
&=  -\frac{\s}{\eps^{\frac{1}{p+1}} |\sigma-s_0|^{\frac{p}{p+1}}}\left(F_{\s \nu}\left(\frac{\sigma-s_0}{\eps},v^{\eps}(\sigma)\right)-\left\langle F\right\rangle (v^\eps(\sigma))\right)  \\
&+\frac{1}{(p+1)|\sigma-s_0|^\frac{p}{p+1}} (\partial_2 \Omega_\nu) \left(\frac{|\sigma-s_0|}{\eps},v^{\eps}(\sigma)\right) \; F_{\mu_\sigma} \left(\frac{\sigma-s_0}{\eps},v^\eps(\sigma)\right),
\end{align*}
that is to say, taking $\nu=\s \mu_\sigma$
\begin{align}
&\frac{1}{|\sigma-s_0|^\frac{p}{p+1}}\left(F_{\mu_\sigma}\left(\frac{\sigma-s_0}{\eps},v^\eps(\sigma)\right)-\left\langle F\right\rangle (v^\eps(\sigma))\right)=-\s\eps^{\frac{1}{p+1}} \frac{d}{d\sigma}\left(\Omega_{\s \mu_\sigma}\left(\frac{|\sigma-s_0|}{\eps},v^\eps(\sigma)\right)\right)\label{R1}\\
&\qquad\qquad \qquad \qquad+  \frac{\s \eps^{\frac{1}{p+1}}}{(p+1)|\sigma-s_0|^\frac{p}{p+1}}\partial_2 \Omega_{\s \mu_\sigma}\left(\frac{|\sigma-s_0|}{\eps},v^\eps(\sigma)\right)F_{\mu_\sigma}\left(\frac{\sigma-s_0}{\eps},v^\eps(\sigma)\right),\nonumber
\end{align}
where we have used \eqref{eq:diffs}. For $\sigma \leq s \leq s_0$ we have $\mu_\sigma=(-1)^p=\mu_s$, $\s=-1$ and therefore
\begin{align} 
R^\eps(s)=&\frac{\eps^{\frac{1}{p+1}}}{p+1} \left(\Omega_{-\mu_s}\left(\frac{s_0-s}{\eps},v^\eps(s)\right)-\Omega_{-\mu_s}\left(\frac{s_0}{\eps},v_0^\eps\right)\right)\label{R1bis}\\
&-\frac{\eps^{\frac{1}{p+1}}}{(p+1)^2}\int_{0}^s\frac{1}{(s_0-\sigma)^\frac{p}{p+1}}\partial_2 \Omega_{-\mu_s}\left(\frac{s_0-\sigma}{\eps},v^\eps(\sigma)\right)F_{-\mu_s}\left(\frac{s_0-\sigma}{\eps},v^\eps(\sigma)\right)d\sigma\nonumber
\end{align}
a relation from which  we may deduce, using \eqref{est1} and Assumption \ref{hyp1},  that $|R^\eps(s)|\leq C\eps^{1/(p+1)}$. In particular, $|R^\eps(s_0)|\leq C\eps^{1/(p+1)}$. As for $s \geq s_0$, we  have  $\mu_\sigma = \s = 1$ and thus
\begin{align} 
R^\eps(s)=& R^\eps(s_0) + \frac{\eps^{\frac{1}{p+1}}}{p+1} \left(\Omega_{1}\left(0,v^\eps(s_0)\right)-\Omega_{1}\left(\frac{s-s_0}{\eps},v^\eps(s)\right)\right)\label{R2bis}\\
&+\frac{\eps^{\frac{1}{p+1}}}{(p+1)^2}\int_{s_0}^s\frac{1}{(\sigma-s_0)^\frac{p}{p+1}}\partial_2 \Omega_{1} \left(\frac{\sigma-s_0}{\eps},v^\eps(\sigma)\right)F_{1} \left(\frac{\sigma-s_0}{\eps},v^\eps(\sigma)\right)d\sigma\nonumber
\end{align}
and we may again conclude from \eqref{est1} and Assumption \ref{hyp1} that $|R^\eps(s)|\leq C\eps^{\frac{1}{p+1}}$ for $s_0 \leq s \leq S$ and eventually for all $0 \leq s \leq S$. Finally, we have on the one hand, 
\begin{equation*}
v^\eps(s)=v_0^{\eps}+\frac{1}{p+1}\int_{0}^s\frac{1}{|\sigma-s_0|^{\frac{p}{p+1}}}\left\langle F\right\rangle (v^\eps(\sigma))d\sigma+ \mathcal O(\eps^{\frac{1}{p+1}}),
\end{equation*}
and on the other hand, 
$$
\underline{v}^\eps(s)=v_0^\eps+\frac{1}{p+1}\int_{0}^s\frac{1}{|\sigma-s_0|^{\frac{p}{p+1}}}\left\langle F\right\rangle (\underline{v}^{\eps}(\sigma))d\sigma,
$$
as long as the solution of \eqref{averagedmodel} exists. Assumption \ref{hyp1} and a standard bootstrap argument based on the Gronwall lemma then enable to conclude.
\end{proof}
\subsection{Next term of the asymptotic expansion}
\label{subsec-next}

This section now presents how the estimate of Proposition \ref{prop1} (analogously Theorem \ref{thm:main}) 
can be refined by introducing an additional term in the asymptotic expansion.

\begin{proposition}
\label{prop2} Let
$\mu=(-1)^p$, and $\delta_{p}=1$ if $p=1$, $\delta_{p}=0$ otherwise. Under Assumption \ref{hyp1}, if we consider the solutions $\bar{v}^\eps$ and $\bar{w}^\eps$ of the averaged equation \eqref{averagedmodel} respectively on $[0,s_0]$ and $[s_0,S]$ and with the respective initial conditions 
\begin{align}
&\bar{v}^\eps(0)=v_0^\eps -  \frac{\eps^{\frac{1}{p+1}}}{p+1} \Omega_{-\mu}\left(\frac{s_0}{\eps},v_0^\eps\right) 
, \label{eq:moy1} \\
&\bar w^\eps(s_0) = \bar{v}^\eps(s_0)+\frac{\eps^{\frac{1}{p+1}}}{p+1} \Big( \Omega_1\left(0,\bar v^\eps(s_0)\right)+ \Omega_{-\mu} \left(0,\bar v^\eps(s_0)\right) \Big)- \frac{\delta_{p} \eps}{4}\log\left(\frac{\eps}{\eps+s_0}\right)\left\langle \partial_2 G \, F \right \rangle(\bar{v}^{\eps}(s_0)), \nonumber \label{eq:moy2}
\end{align}
and $\tilde v^\eps$ the {\em continuous} function defined by the following expressions:
\begin{align*}
s\leq s_0, \; & \widetilde{v}^\eps(s)=\bar{v}^\eps(s)+\frac{\eps^{\frac{1}{p+1}}}{p+1}\Omega_{-\mu} \left(\frac{s_0-s}{\eps},\bar v^\eps(s)\right) -  \frac{\delta_{p} \eps}{4}\log\left(\frac{\eps+s_0-s}{\eps+s_0}\right)\left\langle \partial_2 G F\right\rangle(\bar v^\eps(s)), \\
s_0 \leq s, \; & \widetilde{v}^\eps(s)=\bar{w}^\eps(s)-\frac{\eps^{\frac{1}{p+1}}}{p+1}\Omega_{1} \left(\frac{s-s_0}{\eps},\bar w^\eps(s)\right)+ \frac{\delta_{p} \eps}{4}\log\left(\frac{\eps+s-s_0}{\eps}\right)\left\langle \partial_2 G F\right\rangle(\bar w^\eps(s_0))
+  \beta^\eps
\end{align*}
where 
$$
\beta^\eps=\frac{\eps^{\frac{1}{p+1}}}{p+1}\Omega_1\left(0,\bar w^\eps(s_0)\right)-\frac{\eps^{\frac{1}{p+1}}}{p+1} \Omega_1 \left(0,\bar v^\eps(s_0)\right),
$$
then we have
\begin{equation}\label{cv2}
\forall s \in [0,S], \quad |v^\eps(s)-\widetilde v^\eps(s)|\leq C\,\eps^{\frac{2}{p+1}}.
\end{equation}
\end{proposition}
\begin{proof}
In order to refine estimates (\ref{R1bis}) and (\ref{R2bis}) of $R^\eps(s)$ obtained in the proof of Proposition \ref{prop1}, we rewrite them as 
\begin{align}
s \leq s_0: R^\eps(s)&=\frac{\eps^{\frac{1}{p+1}}}{p+1} \left(\Omega_{-\mu}\left(\frac{s_0-s}{\eps},v^\eps(s)\right)-\Omega_{-\mu}\left(\frac{s_0}{\eps},v_0^\eps\right)
 -  \mathcal{I}_{-\mu}(0,s)\right) , \label{eq:Rpur1}\\ 
s \geq s_0: R^\eps(s)&= R^\eps(s_0) + \frac{\eps^{\frac{1}{p+1}}}{p+1} \left(\Omega_{1}\left(0,v^\eps(s_0)\right)-\Omega_{1}\left(\frac{s-s_0}{\eps},v^\eps(s)\right) +  \mathcal{I}_1(s_0,s) \right), \label{eq:Rpur2}
\end{align}
where the expression of $\mathcal I_{\nu}$  coincides with $\mathcal I$ in Lemma \ref{lem:integralI} for $\phi(\sigma,u) = \partial_2 \Omega_{\nu} F_{\nu}(\sigma,u)$ and $\psi(\sigma,u) = \partial_2 G_{\nu} F_{\nu}(\sigma,u)$. 
If $x$ and $\underline x$ differ by an $\mathcal O(\eps^\frac{1}{p+1})$, then, using (\ref{est1})-(\ref{est2}), one has 
\begin{align*}
 \forall \nu=\pm 1,  \quad \left|\Omega_\nu\left(\frac{s}{\eps},x \right)-\Omega_\nu\left(\frac{s}{\eps},\underline x\right)\right| \leq C\eps^\frac{1}{p+1} 
\end{align*}
and owing to \eqref{cv}, estimates $\bar v^\eps(0)-\underline v^\eps(0) = \mathcal O(\eps^\frac{1}{p+1})$ and  $\bar w^\eps(s_0)-\bar v^\eps(s_0) = \mathcal O(\eps^\frac{1}{p+1})$, and the Gronwall lemma, it stems that 
$$
\forall 0 \leq s \leq s_0, \quad v^\eps(s)-\bar v^\eps(s)=\mathcal O(\eps^\frac{1}{p+1})
\quad \mbox{ and } \quad 
\forall s_0 \leq s \leq S, \quad \bar w^\eps(s)-v^\eps(s) = 
\mathcal O(\eps^\frac{1}{p+1})
$$ 
so that $v^\eps(s)$ can be replaced by $\bar v^\eps(s)$ in \eqref{eq:Rpur1} and by $\bar w^\eps(s)$ in \eqref{eq:Rpur2}, up to $\mathcal O(\eps^\frac{2}{p+1})$-terms. \\ \\
\medskip
{\bf Case $p>1$:} Lemma \ref{lem:integralI} shows that the terms  $\frac{\eps^\frac{1}{p+1}}{p+1} \mathcal I_\nu$ in \eqref{eq:Rpur1} and \eqref{eq:Rpur2} are of order $\mathcal O(\eps^\frac{2}{p+1})$,
we thus have for $s \leq s_0$
\begin{align*}
v^\eps(s)=v_0^\eps+\frac{1}{p+1}\int_{0}^s\frac{\left\langle F\right\rangle (v^\eps(\sigma))}{|\sigma-s_0|^\frac{p}{p+1}}d\sigma 
+\frac{\eps^\frac{1}{p+1}}{p+1} \left[\Omega_{-\mu}
\left(\frac{s_0-\sigma}{\eps},\bar v^\eps(\sigma)\right)\right]_{\sigma=0}^{\sigma=s}  + \mathcal O(\eps^\frac{2}{p+1}),
\end{align*}
that is to say, by denoting $V^\eps(s) =v^\eps(s)-\frac{\eps^\frac{1}{p+1}}{p+1}\Omega_{-\mu}
\left(\frac{s_0-s}{\eps},\bar v^\eps(s)\right)$, 
the equation
\begin{align*}
V^\eps(s)-V^\eps(0)&=\frac{1}{p+1}\int_{0}^s\frac{\left\langle F\right\rangle \Big(V^\eps(\sigma)+(v^\eps(\sigma)-V^\eps(\sigma))\Big)}{(s_0-\sigma)^\frac{p}{p+1}}d\sigma +\mathcal O(\eps^\frac{2}{p+1})\\
&=\frac{1}{p+1} \int_{0}^s\frac{\left\langle F\right\rangle \big(V^\eps(\sigma)\big)d\sigma + \left\langle \partial_2 F\right\rangle(V^\eps(\sigma)) \, (v^\eps(\sigma)-V^\eps(\sigma))}{(s_0-\sigma)^\frac{p}{p+1}} d\sigma+\mathcal O(\eps^\frac{2}{p+1}), \\
&=\frac{1}{p+1} \int_{0}^s\frac{1}{(s_0-\sigma)^\frac{p}{p+1}}\left\langle F\right\rangle \big(V^\eps(\sigma)\big)d\sigma +\mathcal O(\eps^\frac{2}{p+1}),
\end{align*}
where we have used Remark \ref{rem:splusun} to get the bound 
$$
\int_{0}^s \left|\frac{1}{(s_0-\sigma)^\frac{p}{p+1}} \left\langle  \partial_2 F\right\rangle(V^\eps(\sigma)) \Omega_{-\mu} 
\left(\frac{s_0-\sigma}{\eps},\bar v^\eps(\sigma)\right)\right| d\sigma \leq C \eps^\frac{1}{p+1} \int_{0}^{+\infty} \frac{1}{(\sigma(1+\sigma))^\frac{p}{p+1}}  d\sigma.
$$
From $V^\eps(0)-\bar v^\eps(0) = {\cal O}(\eps^{\frac{2}{p+1}})$ and equation \eqref{averagedmodel}, we obtain by the Gronwall lemma 
$$
\forall s \leq s_0, \quad \left|\tilde v^\eps(s)-v^{\eps}(s)\right| = 
\left|V^\eps(s)-\bar v^\eps(s)\right|\leq C\eps^{\frac{2}{p+1}}.
$$
As for $s \geq s_0$,  we write 
\begin{align*}
v^\eps(s)&
= v^\eps(s_0) +\frac{1}{p+1}\int_{s_0}^s\frac{\left\langle F\right\rangle (v^\eps(\sigma))}{(\sigma-s_0)^\frac{p}{p+1}}d\sigma + (R^\eps(s)-R^\eps(s_0)) \\
&=v^\eps(s_0) + \frac{1}{p+1}\int_{s_0}^s\frac{\left\langle F\right\rangle (v^\eps(\sigma))}{(\sigma-s_0)^\frac{p}{p+1}}d\sigma - \frac{\eps^{\frac{1}{p+1}}}{p+1} \left[\Omega_{1}\left(\frac{\sigma-s_0}{\eps},\bar w^\eps(\sigma)\right)\right]_{\sigma=s_0}^{\sigma=s} + \mathcal O(\eps^\frac{2}{p+1}) 
\end{align*}
that is to say, by denoting $W^\eps(s) = v^\eps(s) + \frac{\eps^\frac{1}{p+1}}{p+1}\Omega_{1}
\left(\frac{s-s_0}{\eps},\bar w^\eps(s)\right)$, 
the simple equation
\begin{align*}
W^\eps(s)
&= W^\eps(s_0) +\frac{1}{p+1}\int_{s_0}^s\frac{\left\langle F\right\rangle \Big(W^\eps(\sigma))\Big)}{(\sigma-s_0)^\frac{p}{p+1}}d\sigma +\mathcal O(\eps^\frac{2}{p+1})
\end{align*}
and by comparing with equation \eqref{averagedmodel}, the Gronwall lemma enables to conclude that $W^\eps(s)-\bar w^\eps(s) = \mathcal O(\eps^\frac{2}{p+1})$ given that $W^\eps(s_0)- \bar w^\eps(s_0) = \mathcal O(\eps^\frac{2}{p+1})$ (by definition of $\bar w^\eps(s_0)$ and $W^\eps(s_0)$ and estimate \eqref{cv2} for $s=s_0$). The statement for $s \geq s_0$ now follows from  $\beta^\eps=\mathcal O(\eps^\frac{2}{p+1})$. \\ \\
\medskip
{\bf Case $p=1$:} 
This case differs in that  the terms $\frac{\sqrt{\eps}}{2} \mathcal I_\nu$ in \eqref{eq:Rpur1} and \eqref{eq:Rpur2} are now of order $\eps \log(\eps)$ for $s$ close to $s_0$. This yields for $s \leq s_0$
\begin{align*}
v^\eps(s)=v_0^\eps+\frac{1}{2}\int_{0}^s\frac{\left\langle F\right\rangle (v^\eps(\sigma))}{\sqrt{s_0-\sigma}}d\sigma 
+\frac{\sqrt{\eps}}{2}\Omega_{-\mu}
\left[\left(\frac{s_0-\sigma}{\eps},\bar v^\eps(\sigma)\right)\right]_{\sigma=0}^{\sigma=s}  - \frac{\sqrt{\eps}}{2} \mathcal{I}_{-\mu}(0,s) + \mathcal O(\eps),
\end{align*}
that is to say, by denoting 
$$
V^\eps(s) =v^\eps(s)-\frac{\sqrt{\eps}}{2}\Omega_{-\mu}
\left(\frac{s_0-s}{\eps},\bar v^\eps(s)\right)+ \frac{\eps}{4}\log\left(\frac{\eps+s_0-s}{\eps+s_0}\right)\left\langle  \partial_2 G \, F\right\rangle(\bar v^\eps(s)),
$$
the equation
\begin{align*}
V^\eps(s)&=V_0^\eps + \int_{0}^s\frac{\left\langle F\right\rangle \big(V^\eps(\sigma)\big)}{2 \sqrt{s_0-\sigma}}d\sigma + \int_{0}^s\frac{\left\langle \partial_2 F\right\rangle(V^\eps(\sigma)) }{2 \sqrt{s_0-\sigma}}\, (v^\eps(\sigma)-V^\eps(\sigma)) d\sigma+\mathcal O(\eps) \\
&=V_0^\eps + \int_{0}^s\frac{\left\langle F\right\rangle \big(V^\eps(\sigma)\big)}{2 \sqrt{s_0-\sigma}}d\sigma + \frac{\sqrt{\eps}}{4} \int_{0}^s\frac{\left\langle \partial_2 F\right\rangle(V^\eps(\sigma)) }{\sqrt{s_0-\sigma}}\,\Omega_{-\mu}
\left(\frac{s_0-\sigma}{\eps},\bar v^\eps(\sigma)\right) d\sigma\\
&  - \frac{\eps}{8} \int_{0}^s\frac{\log\left(\frac{\eps+s_0-\sigma}{\eps+s_0}\right)}{\sqrt{s_0-\sigma}}\,\left\langle \partial_2 F\right\rangle(V^\eps(\sigma))  \left\langle  \partial_2 G \, F\right\rangle(\bar v^\eps(\sigma)) + \mathcal O(\eps) \\
& = V_0^\eps + \int_{0}^s\frac{\left\langle F\right\rangle \big(V^\eps(\sigma)\big)}{2 \sqrt{s_0-\sigma}}d\sigma + \mathcal O(\eps),
\end{align*}
where we have used Lemma \ref{lem:integralI} again now with 
$\phi(\sigma,u) = \langle \partial_2 F  \rangle(u) \, \Omega_{-\mu}
\left(\sigma,u\right)$ and $\psi(\sigma,u) = \langle \partial_2 F  \rangle(u) \, G_{-\mu} \left(\sigma,u\right)$, and noticed that 
$\langle \psi \rangle = \langle \partial_2 F  \rangle \, \langle G_{-\mu} \rangle = 0$, to get rid of the second term of the second line. The third term may be bounded through an integartion by parts. We again conclude by the Gronwall lemma. As for $s \geq s_0$, we get 
\begin{align*}
v^\eps(s)=v^\eps(s_0)+\frac{1}{2}\int_{s_0}^s\frac{\left\langle F\right\rangle (v^\eps(\sigma))}{\sqrt{\sigma-s_0}}d\sigma 
-\frac{\sqrt{\eps}}{2}\left[\Omega_{1}
\left(\frac{\sigma-s_0}{\eps},\bar w^\eps(\sigma)\right)\right]_{\sigma=s_0}^{\sigma=s}  + \frac{\sqrt{\eps}}{2} \mathcal{I}_{1}(s_0,s) + \mathcal O(\eps),
\end{align*}
that is to say, by denoting 
$$
W^\eps(s) =v^\eps(s)+\frac{\sqrt{\eps}}{2}\Omega_{1}
\left(\frac{s-s_0}{\eps},\bar w^\eps(s)\right)- \frac{\eps}{4}\log\left(\frac{\eps}{\eps+s-s_0}\right)\left\langle  \partial_2 G \, F\right\rangle(\bar w^\eps(s_0)),
$$
the equation
\begin{align*}
W^\eps(s)&=W^\eps(s_0) + \int_{s_0}^s\frac{\left\langle F\right\rangle \big(W^\eps(\sigma)\big)}{2 \sqrt{\sigma-s_0}}d\sigma - \frac{\sqrt{\eps}}{4} \int_{s_0}^s\frac{\left\langle \partial_2 F\right\rangle(W^\eps(\sigma)) }{\sqrt{\sigma-s_0}}\,\Omega_{1}
\left(\frac{\sigma-s_0}{\eps},\bar w^\eps(\sigma)\right) d\sigma\\
&  + \frac{\eps}{8} \int_{s_0}^s\frac{\log\left(\frac{\eps}{\eps+s-s_0}\right)}{\sqrt{\sigma-s_0}}\,\left\langle \partial_2 F\right\rangle(W^\eps(\sigma))  \left\langle  \partial_2 G \, F\right\rangle(\bar w^\eps(s_0)) + \mathcal O(\eps) \\
& = W^\eps(s_0) + \int_{s_0}^s\frac{\left\langle F\right\rangle \big(W^\eps(\sigma)\big)}{2 \sqrt{\sigma-s_0}}d\sigma + \mathcal O(\eps),
\end{align*}
where we have used equation \eqref{eq:I1bis} of Lemma \ref{lem:integralI}, and we may conclude as before.
\end{proof}
\begin{corollary} \label{coro}
Let $\mu=(-1)^p$, $\delta_{p}=1$ if $p=1$, $\delta_{p}=0$ otherwise and $\tau_0=\frac{t_0}{\eps}$. Under Assumption \ref{hyp1},   consider $\bar{u}^\eps$,   the solution of 
\begin{align} \label{eq:master}
\dot{\bar{u}}^\eps(t) = \langle F \rangle (\bar{u}^\eps(t)), \quad t \in [0,t_0[ \, \cup \, [t_0,T],
\end{align}
with initial conditions $\bar{u}^\eps(0)=u_0^\eps -  \frac{\eps^{\frac{1}{p+1}}}{p+1} \Omega_{-\mu}\left(\frac{t_0^{p+1}}{\eps},u_0^\eps\right)$ and
$$
\bar u^\eps(t_0) = \bar u^\eps_{t_0}+\frac{\eps^{\frac{1}{p+1}}}{p+1} \Big( \Omega_1\left(0,\bar u^\eps_{t_0}\right)+ \Omega_{-\mu} \left(0,\bar u^\eps_{t_0}\right) \Big)+ \frac{\delta_{p} \eps}{4}\log\left(1+\tau_0\right)\left\langle \partial_2 G \, F \right \rangle(\bar u^\eps_{t_0})
$$
where $\bar u^\eps_{t_0}=\lim_{t \rightarrow t_0} \bar u^\eps(t)$ if $t_0>0$ and $\bar u^\eps_{t_0}=\bar{u}^\eps(0)$ if $t_0=0$. 
Then we have 
\begin{equation}\label{cvu}
\forall t \in [0,T], \quad |u^\eps(t)-\widetilde u^\eps(t)|\leq C\,\eps^{\frac{2}{p+1}}
\end{equation}
where $\tilde u^\eps$ is the continuous function defined by the following expressions:
\begin{align*}
0 \leq t \leq t_0: \quad & \widetilde{u}^\eps(t)=\bar{u}^\eps(t)+\frac{\eps^{\frac{1}{p+1}}}{p+1}\Omega_{-\mu} \left(\tau,\bar u^\eps(t)\right) -  \frac{\delta_{p} \eps}{4}\log\left(\frac{1+\tau}{1+\tau_0}\right)\left\langle \partial_2 G F\right\rangle(\bar u^\eps(t)), \\
t_0 \leq t \leq T: \quad & \widetilde{u}^\eps(t)=\bar{u}^\eps(t)-\frac{\eps^{\frac{1}{p+1}}}{p+1}\Omega_{1} \left(\tau,\bar u^\eps(t)\right)+ \frac{\delta_{p} \eps}{4}\log\left(1+\tau\right)\left\langle \partial_2 G F\right\rangle(\bar u^\eps(t_0)) +  \beta^\eps,
\end{align*}
with $\tau=\frac{|t-t_0|^{p+1}}{\eps}$ and $\beta^\eps=\frac{\eps^{\frac{1}{p+1}}}{p+1}\Omega_1\left(0,\bar u^\eps(t_0)\right)-\frac{\eps^{\frac{1}{p+1}}}{p+1} \Omega_1 \left(0,\bar u^\eps_{t_0}\right)$. 
\end{corollary}

\section{A micro-macro method for the case of a multiplicity $p=1$}
\label{sect-micmac}
In this section, 
we suggest a micro-macro decomposition analogous to the one introduced in \cite{CLM17}
and elaborated from the asymptotic analysis of Section \ref{sec-asympt}. In a second step, we propose a second-order {\it uniformly accurate} numerical method derived from this decomposition.
\subsection{The decomposition method}
\label{sect-decomp}
Let $u^\eps(t)$ be the solution of \eqref{eq-ueps} and let $\widetilde u^\eps(t)$ be the approximation defined in Corollary \ref{coro}, and consider the defect function
\begin{equation}\label{defz}
\Delta^\eps(t)=u^\eps(t)-\widetilde u^\eps(t),\qquad \mbox{for }t\in[0,T].\end{equation}
\begin{proposition}
\label{prop3} Assume that $f$ is of class $C^2$ and consider the solution $u^\eps(t)$ of \eqref{eq-ueps} on $[0,T]$.
The function $\Delta^\eps(t)$ defined by \eqref{defz} satisfies
\begin{equation}
\label{estiz}
\forall t\in[0,T],\qquad |\Delta^\eps(t)|\leq C\eps,\end{equation}
\begin{equation}
\label{estiz2}
\forall t\in[0,t_0[\cup]t_0,T],\qquad \left|\dot \Delta^\eps (t)\right|\leq C \sqrt{\eps},\qquad \left|\ddot \Delta^\eps (t)\right|\leq C.
\end{equation}
\end{proposition}
\begin{proof}
By construction, $\tilde u^\eps$ is continuous on $[0,T]$ and estimate \eqref{estiz} is nothing but \eqref{cvu}.  However, its derivatives are not continuous at $t_0$.  Hereafter, it is enough to consider $t$ in  $[0,t_0[$ as the same arguments can be repeated for values in $]t_0,T]$. From the expression of
\begin{align*}
\widetilde{u}^\eps(t)=\bar{u}^\eps(t)+\frac{\eps^{\frac{1}{2}}}{2}\Omega_{1} \left(\tau,\bar u^\eps(t)\right) -  \frac{ \eps}{4}\log\left(\frac{1+\tau(t)}{1+\tau_0}\right)\left\langle \partial_2 G F\right\rangle(\bar u^\eps(t)), \quad \tau(t) = \frac{(t-t_0)^2}{\eps},
\end{align*}
it stems by definition of $\Omega$ (see \eqref{defOmega}) that 
\begin{align}
\dot \Delta^\eps(t) =&\,F\left(\tau,u^\eps\right)-F\left(\tau,\bar u^\eps\right)-\frac{\sqrt{\eps}}{2} \partial_2 \Omega_1\left(\tau,\bar u^\eps\right)\langle F\rangle (\bar u^\eps)\nonumber\\
&-\frac{\sqrt{\eps}}{2} \frac{\sqrt{\tau}}{1+\tau}\left\langle \partial_2 G \, F\right\rangle(\bar u^\eps)+\frac{\eps}{4}\log\left(\frac{1+\tau}{1+\tau_0}\right) \frac{d}{dt}\Big(\left\langle \partial_2 G \, F\right\rangle(\bar u^\eps)\Big),\label{derz}
\end{align}
where we have omitted $t$ in $u^\eps(t)$, $\bar u^\eps(t)$ and $\tau(t)$. From Prop. \ref{prop1} and Eq. \eqref{est2}, we have
$$
\left|F\left(\tau,u^\eps\right)-F\left(\tau,\bar u^\eps\right)\right| \leq C \sqrt{\eps} \quad \mbox{ and } \quad \left|\frac{\sqrt{\eps}}{2}\partial_2 \Omega_1\left(\tau,\bar u^\eps\right)\langle F\rangle (\bar u^\eps)\right|\leq C \sqrt{\eps}.
$$
Besides, $2 \sqrt{\tau} \leq 1+\tau$, $|\eps \log \eps| \leq \sqrt{\eps}$, and the first estimate of \eqref{estiz2} is thus proven. Now, using again equations \eqref{eq-ueps} and \eqref{averagedmodel}, a second derivation leads to 
\begin{align*}
& \ddot \Delta^\eps (t)=-\frac{2 \sqrt{\tau}}{\sqrt{\eps}}\Big(\partial_1 F\left(\tau,u^\eps\right)-\partial_1 F\left(\tau,\bar u^\eps\right)\Big) +\partial_2F\left(\tau,u^\eps\right)F\left(\tau,u^\eps\right) \\
&-2 \partial_2F\left(\tau,\bar u^\eps\right)\langle F\rangle (\bar u^\eps) +\langle \partial_2F\rangle (\bar u^\eps)\langle F\rangle (\bar u^\eps) -\frac{\sqrt{\eps}}{2} \partial_u^2 \Omega_1 \left(\tau,\bar u^\eps\right) \Big( \langle F\rangle (\bar u^\eps),\langle F\rangle (\bar u^\eps)\Big) \\
& -\frac{\sqrt{\eps}}{2} \partial_2\Omega_1\left(\tau,\bar u^\eps\right)\langle \partial_2F\rangle (\bar u^\eps) \, \langle F\rangle (\bar u^\eps) + \frac{1-\tau}{2(1+\tau)^2}\left\langle \partial_2GF\right\rangle(\bar u^\eps) \\
&-\frac{\sqrt{\eps}}{2} \frac{\sqrt{\tau}}{1+\tau}\left\langle \partial_2 G \, F\right\rangle(\bar u^\eps)+\frac{\eps}{4}\log\left(\frac{1+\tau}{1+\tau_0}\right) \frac{d^2}{dt^2}\Big(\left\langle \partial_2 G \, F\right\rangle(\bar u^\eps)\Big).
\end{align*}
Thanks to Assumption \ref{hyp1}, Lemma \ref{lemOmega} and \eqref{averagedmodel}, all the terms are clearly uniformly bounded, except the critical one in the first line, which requires more attention. We get
\begin{align*}
\left|\frac{ \sqrt{\tau}}{\sqrt{\eps}}\Big(\partial_1 F\left(\tau,u^\eps\right)-\partial_1 F\left(\tau,\bar u^\eps\right)\Big)\right|\leq C\frac{|t-t_0|}{\eps}|u^\eps-\bar u^\eps|
\leq C\frac{|t-t_0|}{\eps}|\bar u^\eps-\widetilde u^\eps|+C,
\end{align*}
where we have used the result of Proposition \ref{prop2}, i.e. $|u^\eps-\widetilde u^\eps|\leq C \eps$.
It remains, using the expression of $\widetilde u^\eps$,  to observe that for $t \neq t_0$, $0 < \tau \leq \tau_0$ so that owing to \eqref{est2}, we obtain 
\begin{align*}
\frac{\sqrt{\tau}}{\sqrt{\eps}}|\widetilde u^\eps-\bar u^\eps|&\leq \frac{\sqrt{\tau}}{\sqrt{\eps}}\left(\frac{\sqrt{\eps}}{2}\left|\Omega_1\left(\tau,\bar u^\eps\right)\right|+\frac{\eps}{4}\left|\log\left(\frac{1+\tau}{1+\tau_0}\right)\right|\left|\left\langle \partial_2GF\right\rangle(\bar u^\eps(t_0))\right|\right)\\
&\leq C\frac{\sqrt{\tau}}{\sqrt{\eps}}\frac{\sqrt{\eps }}{\sqrt{\tau}}+C \sqrt{\frac{\tau}{\tau_0}} \left|\log\left(\frac{1+\tau}{1+\tau_0}\right)\right|\leq C.
\end{align*}
This completes the  proof.
\end{proof}
\subsection{A uniformly accurate second order numerical method}
We are now in position to introduce a uniformly accurate second-order numerical scheme.
Consider 
$0=t^{[0]}<\ldots <t^{[k]}<\cdots<t^{[N]}=T$ a subdivision of $[0,T]$ with $h=\max_{k=1,\ldots,N}(t^{[k]}-t^{[k-1]})$ and  assume that $t_0$ is one of the discretization points, i.e. $t_0=t^{[k_0]}$ for some $k_0$. Our scheme  provides  approximations $(\bar u^k,\Delta^k)$ of the pair $(\bar u^\eps(t^{[k]}),\Delta^\eps(t^{[k]}))$. An approximation $u^k$ of $u^\eps(t^{[k]})$ is then derived by assembling the approximation $\widetilde u^k$ of $\widetilde u^\eps(t^{[k]})$ from formulas in Corollary  \ref{coro} and eventually by setting $u^k=\widetilde u^k+\Delta^k$. Given that problem \eqref{eq:master} is nonstiff, any second-order numerical scheme is suitable for the computation of $\bar u^k$ and thus of $\widetilde u^k$, 
and we simply choose here the Heun method
$$
\bar u^{k+1} = \bar u^{k} + \frac h2 \langle F \rangle(\bar u^{k}) + \frac h2 \langle F \rangle
\left(\bar u^{k} + h \langle F \rangle(\bar u^{k})\right).
$$
As a consequence, we limit ourselves to the scheme for $\Delta^\eps$. Starting from
\begin{equation}
\label{eqzk}\Delta^\eps(t^{[k+1]})=\Delta^\eps(t^{[k]})+\int_{t^{[k]}}^{t^{[k+1]}} \hspace{-0.3cm} F\left(\tau(\xi),\widetilde u^\eps(\xi)+\Delta^\eps(\xi)\right)d\xi-(\widetilde u^\eps(t^{[k+1]})-\widetilde u^\eps(t^{[k]})), 
\end{equation}
where $\tau(\xi)=\frac{|\xi-t_0|^{p+1}}{\eps}$,
we consider at time $t^{[k+1/2]}=\frac{t^{[k]}+t^{[k+1]}}{2}$ the approximation
$$
\Delta^{k+\frac12}=\Delta^k+\int_{t^{[k]}}^{t^{[k+\frac12]}} \hspace{-0.3cm}F\left(\tau(\xi),\widetilde u^k+\Delta^k\right)d\xi-(\widetilde u^{k+\frac12}-\widetilde u^k).$$
Since the function $\widetilde u^\eps+\Delta^\eps=u^\eps$ has a bounded first time-derivative, the error associated to this scheme is of order $\mathcal O(h^2)$. Expanding $F$ in Fourier series, we see that the scheme necessitates the computation of integrals of terms of the form $e^{i\ell \xi^2}$ which may be easily computed numerically using the complex {\tt erf} function. Now, for $k<k_0$ and $t\leq t_0$, we identify the smooth part of $u^\eps(t)$ as
$$
a^\eps(t)=\bar u^\eps(t)+\Delta^\eps(t)-\frac{\eps}{4}\log\left(\frac{1+\tau}{1+\tau_0}\right)\left\langle \partial_2GF\right\rangle(\bar u^\eps(t)),$$
so that 
$$
u^\eps(t)=\widetilde u^\eps(t)+\Delta^\eps(t)=a^\eps(t)+\frac{\sqrt{\eps}}{2}\Omega_1\left(\tau(t),\bar u^\eps(t)\right)$$
and, by Proposition \ref{prop3} and its proof, it is clear that the second time-derivative of $a^\eps$ is uniformly bounded.
In order to approximate \eqref{eqzk}, we remark that
$$a^\eps(\xi)=a^k+\frac{a^{k+1/2}-a^k}{t^{[k+1/2]}-t^{[k]}}\left(\xi-t^{[k]}\right)+\mathcal O(h^2),$$
where setting $\bar u^{k+1/2} = \bar u^{k} + \frac h2 \langle F \rangle (\bar u^{k})$, we define
for $\tau^{[k+1/2]}=\tau(t^{[k+1/2]})$,
$$
a^{k+1/2}=\bar
u^{k+1/2}+\Delta^{k+1/2}-\frac{\eps}{4}\log\left(\frac{1+\tau^{[k+1/2]}}{1+\tau_0}\right)\left\langle
\partial_2GF\right\rangle(\bar u^{k+1/2}).
$$
Moreover, we have
$$\forall (s,\hat s) \in \R_+^2,\qquad\left|\Omega_1(s,\bar u^k)-\Omega_1(\hat s,\bar u^k)\right|=\left|\int_s^{\hat s}\frac{F(\sigma,\bar u^k)-\langle F\rangle(\bar u^k)}{\sqrt{\sigma}}d\sigma\right|\leq C|\sqrt{\hat s}-\sqrt{s}|$$
so that 
$$\left|\frac{\sqrt{\eps}}{2}\Omega_1\left(\tau(\xi),\bar u^k\right)-\frac{\sqrt{\eps}}{2}\Omega_1\left(\tau(t^{[k]}),\bar u^k\right)\right|\leq Ch$$
and
$$\frac{\sqrt{\eps}}{2}\Omega_1\left(\tau(\xi),\bar u^\eps(\xi)\right)=\frac{\sqrt{\eps}}{2}\Omega_1\left(\tau(\xi),\bar u^k\right)+\frac{\sqrt{\eps}}{2}\left(\xi-t^{[k]}\right)\partial_2\Omega_1\left(\tau(\xi),\bar u^k\right)\langle F\rangle (\bar u^k)+\mathcal O(h^2).$$
Therefore, denoting
$$b^k=a^k+\frac{\sqrt{\eps}}{2}\Omega_1\left(\tau(t^{[k]}),\bar u^k\right),$$
our numerical scheme takes the form
\begin{align*}
\Delta^{k+1}=&\Delta^k+\int_{t^{[k]}}^{t^{[k+1]}}F\left(\tau(\xi),b^k\right)d\xi+\int_{t^{[k]}}^{t^{[k+1]}}\left(\xi-t^{[k]}\right)\partial_2F\left(\tau(\xi),b^k\right)\frac{a^{k+1/2}-a^k}{t^{[k+1/2]}-t^{[k]}}d\xi\\
&+\int_{t^{[k]}}^{t^{[k+1]}}\frac{\sqrt{\eps}}{2}\left(\xi-t^{[k]}\right)\partial_2F\left(\tau(\xi),b^k\right)\partial_2\Omega_1\left(\tau(\xi),\bar u^k\right)\langle F\rangle (\bar u^k)d\xi\\
& +\int_{t^{[k]}}^{t^{[k+1]}}\frac{\sqrt{\eps}}{2}\partial_2F\left(\tau(\xi),b^k\right)
\left(\Omega_1\left(\tau(\xi),\bar u^k\right)-\Omega_1\left(\tau(t^{[k]}),\bar u^k\right)\right)d\xi+\widetilde u^\eps(t^{[k+1]})-\widetilde u^\eps(t^{[k]}),
\end{align*}
and has a truncation error of size $\mathcal O(h^3)$, uniformly in $\eps$. As for $k\geq k_0$, we have
\begin{align*}
a^k=\bar u^k+\Delta^k+\frac{\eps}{4}\log\left(\frac{1+\tau}{1+\tau_0}\right)\left\langle \partial_2GF\right\rangle(\bar u^\eps(t_0))+\beta^\eps, 
\quad b^k=a^k-\frac{\eps^{1/2}}{2}\Omega_1\left(\tau(t^k),\bar u^{k}\right),
\end{align*}
and 
\begin{align*}
\Delta^{k+1}=&\Delta^k+\int_{t^{[k]}}^{t^{[k+1]}}F\left(\tau(\xi),b^k\right)d\xi +\int_{t^{[k]}}^{t^{[k+1]}}\left(\xi-t^{[k]}\right)\partial_2F\left(\tau(\xi),b^k\right)\frac{a^{k+1/2}-a^k}{t^{[k+1/2]}-t^{[k]}}d\xi\\
&-\int_{t^{[k]}}^{t^{[k+1]}}\frac{\sqrt{\eps}}{2}\left(\xi-t^{[k]}\right)\partial_2F\left(\tau(\xi),b^k\right)\partial_2\Omega_1\left(\tau(t^{[k]}),\bar u^{k}\right)\langle F\rangle (\bar u^{k})d\xi\\
&-\int_{t^{[k]}}^{t^{[k+1]}}\frac{\sqrt{\eps}}{2}\partial_2F\left(\frac{\Gamma(\tau)}{\eps},b^k\right)\left(\Omega_1\left(\tau(\xi),\bar u^k\right)-\Omega_1\left(\tau(t^{[k]}),\bar u^{k}\right)\right)d\tau \\
&+\widetilde u^\eps(t^{[k+1]})-\widetilde u^\eps(t^{[k]}).
\end{align*}
According to the above computations, the uniform accuracy with second order of the proposed scheme may now be stated:
\begin{proposition} \label{propscheme}
Assume that $f$ is of class $C^2$. Consider the solution $u^\eps(t)$ of \eqref{eq-ueps}
on $[0,T]$, and the numerical scheme $(\widetilde u^k,\Delta^k)$ defined above.
Then $u^k=\widetilde u^k+\Delta^k$ yields a uniformly accurate approximation of the solution
$u^\eps(t_k)$. Precisely, there exist $\eps_0>0$ and $h_0>0$ such that for all $\eps\leq \eps_0$ and all $h\leq h_0$,
$$
|u^k - u^\eps(t^{[k]})| \leq C h^2
$$
for all $t^{[k]}\leq T$ and where $C$ is independent of $\eps$ and $h$.
\end{proposition}
\subsection{Numerical experiments}
For the sake of simplicity, we test our method on the H\'enon-Heiles system $U^\eps=(q_1,q_2,p_1,p_2)$ with a time-varying parameter $\gamma(t)=2(t-t_0)$,
$$ 
\dot U^\eps(t) = 
\left( \frac{\gamma(t)}{\eps} p_1, p_2, -\frac{\gamma(t)}{\eps}q_1-2 q_1 q_2, -q_2-q_1^2+q_2^2 \right), \; U^\eps(0)=(0.9,0.6,0.8,0.5).
$$
The associated filtered system, satisfied by the variable $u^\eps(t)\in \RR^4$ defined by
$$
u^\eps(t)=\left(\cos(\theta)q_1(t)-\sin(\theta) \, p_1(t),q_2(t),\sin(\theta) q_1(t)+\cos(\theta) p_1(t),p_2(t)\right),
$$
with $\theta=\frac{t_0^2+(t-t_0)|t-t_0|}{\eps}$, takes the form \eqref{eq-ueps} with
\begin{align*}
F_1(\theta,u) &= 2\sin \theta\left(u_1\cos\theta+u_3\sin\theta\right)u_2, \quad \; \; \; F_2(\theta,u) = u_4, \\
F_3(\theta,u) & = -2\cos \theta\left(u_1\cos\theta+u_3\sin\theta\right)u_2, \quad F_4(\theta,u) = -\left(u_1\cos\theta+u_3\sin\theta\right)^2+u_2^2-u_2).
\end{align*}
We consider a time interval of length $T=1$ and take $t_0=1/3$ as time where the oscillatory frequency vanishes.
The reference solution is obtained using the matlab {\tt ode45} routine with a tiny tolerance.
On Figure \ref{fig2}, we have represented the maximal error along the time interval
of the numerical solution. 
On the left picture,  the error is plot
as a function of the stepsize $h$, for fixed values $\eps\in\{2^{-k},\,k=0,\cdots, 11\}$, while on the right picture, the error is plot as a function of $\eps$, for fixed values $h\in\{0.1/2^{-k},\,k=0,\cdots, 9\}$. All curves are in perfect agreement with Proposition \ref{propscheme}.
\begin{figure}[H]
\centering
\includegraphics[width=0.52\textwidth]{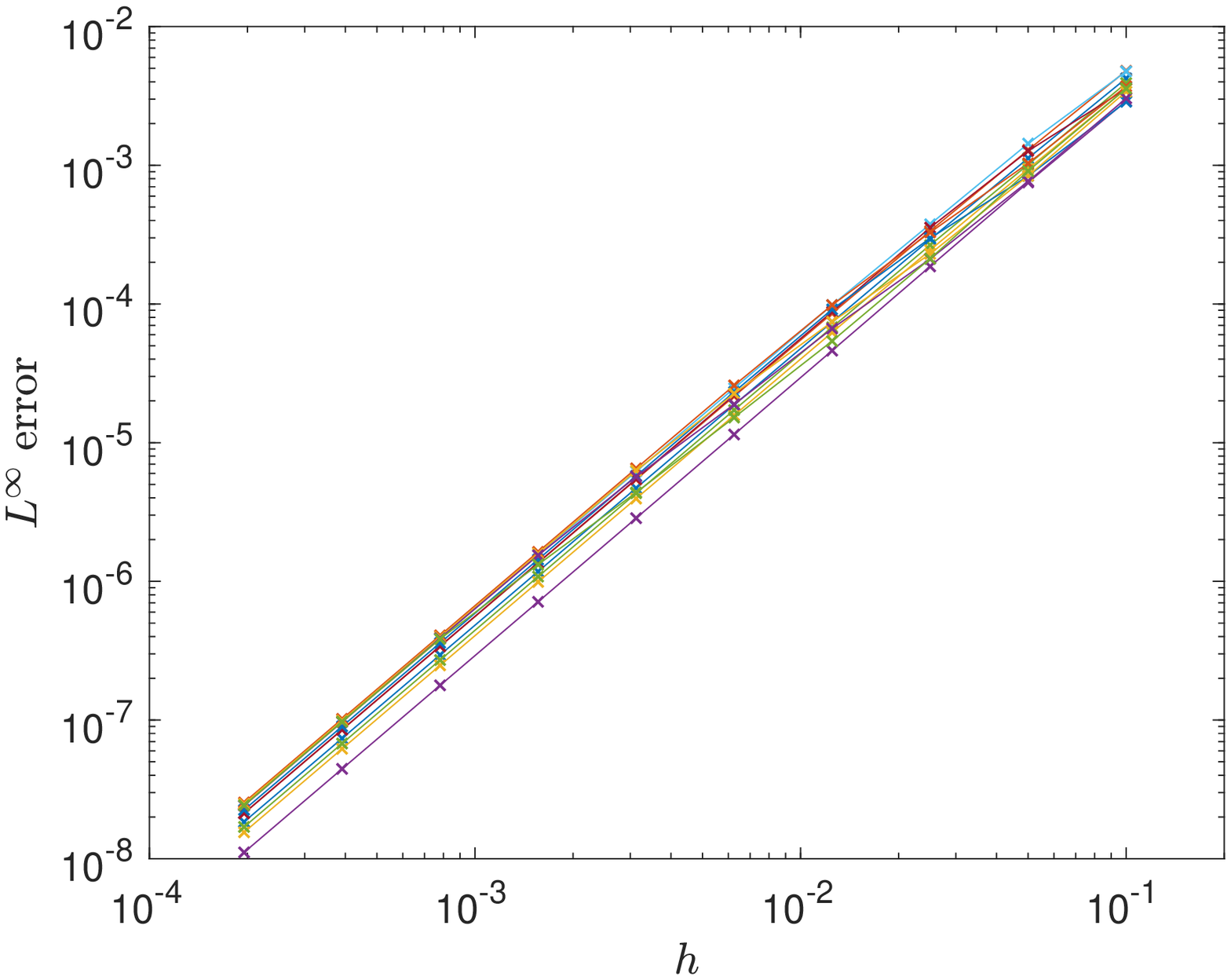}
\hspace{-0.9cm} \includegraphics[width=0.52\textwidth]{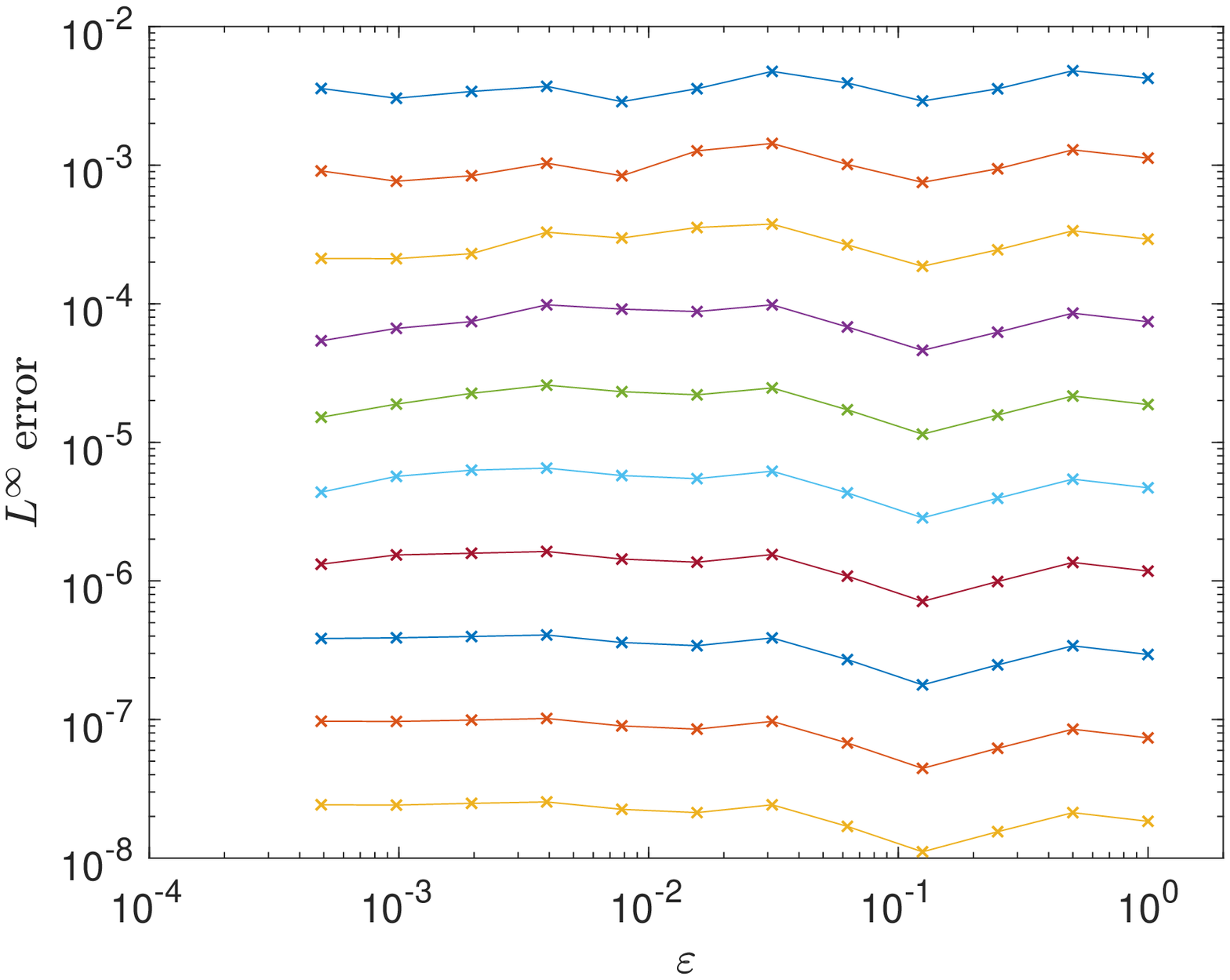}
\caption{Error as a function of $h$ for $\eps\in\{2^{-k},\,k=0,\cdots, 11\}$ (left) and error as a function of $\eps$ for $h\in\{0.1/2^{-k},\,k=0,\cdots, 9\}$ (right).}\label{fig2}

\end{figure}

\bigskip

\noindent \textbf{Acknowledgements.}
The work of P.C., M.L., and F.M. is partially supported by the ANR project Moonrise ANR-14-CE23-0007-01.
The work of G.V. is partially supported by the Swiss National Science Foundation, grants No: 200020\_178752 and 200021\_162404.

\bibliographystyle{alpha}
\bibliography{biblioNCF}
 \end{document}